\documentclass[12pt,a4paper, oneside, bold,secthm,seceqn,amsthm,ussrhead,reqno]{article}
\usepackage[utf8]{inputenc}
\usepackage[english]{babel}
\usepackage[symbol]{footmisc}
\usepackage{amssymb,amsmath,amsthm,amsfonts,xcolor,enumerate,hyperref, comment,longtable, cleveref}
\usepackage{verbatim}
\usepackage{times}
\usepackage{cite}
\usepackage{pdflscape}
\usepackage{ulem}        
\usepackage[mathcal]{euscript}
\usepackage{tikz}
\usepackage{hyperref}
 
\usepackage{cancel}
\usepackage{stmaryrd}
 
\usepackage{amsfonts}
\usepackage{amssymb}
\usepackage{times}
\usepackage{xcolor}
\usepackage{tkz-graph}
\usepackage{url}
\usepackage{float}
\usepackage{tasks}
\usepackage{array}

\let\<\langle
\let\>\rangle
 
\usepackage{cite}
\usepackage{hyperref}

 \usepackage{fancyhdr} 
\usepackage{amsfonts}
\usepackage{amsmath}
\usepackage{eurosym}
\usepackage{geometry}

\usepackage{caption,booktabs}

\theoremstyle{plain}
\newtheorem{theorem}{Theorem}
\newtheorem{lemma}[theorem]{Lemma}
\newtheorem{proposition}[theorem]{Proposition}

\newtheorem{example}[theorem]{Example}
\newtheorem{definition}[theorem]{Definition}
\newtheorem{corollary}[theorem]{Corollary}

\def\As{\mathrm{As}}
\def\Com{\mathrm{Com}}
\def\di{\mathrm{di}}

\def\Leib{\mathrm{Leib}}
\def\Nov{\mathrm{Nov}}
\def\Var{\mathrm{Var}}

\def\Alia{\mathrm{Alia}}
\def\Zinb{\mathrm{Zinb}}
\def\Alt{\mathrm{Alt}}

\usetikzlibrary{arrows}

\usepackage{fouriernc}
\usepackage{backref}

\renewcommand*{\backrefalt}[4]{%
    \ifcase #1 %
        \else [p. #2]%
    \fi%
}

\begin{document}

\noindent{\Large
A unified approach to the definition  of anti-algebras}
 \footnote{
This research was funded by the Science Committee of the Ministry of Science and Higher Education of the Republic of Kazakhstan $\big($Grant No. AP26199089$\big)$. 
The authors thank the Centre International de Rencontres Mathématiques $\big($CIRM$\big)$, where the first part of this work was prepared during their research stay "3473 - Structural problems of non-quadratic operads" $\big($17-28 November 2025$\big)$ and the SRMC (Sino-Russian Mathematics Center in Peking University, Beijing, China), where 
second part of this work was prepared.
The work is supported by 
FCT  2023.08031.CEECIND, UID/00212/2025, and 1347597976/2024-25. The authors are grateful to Abdenacer Makhlouf for insightful discussions that drew their attention to the topic under consideration.}

\begin{center}

 {\bf
   Ivan Kaygorodov\footnote{CMA-UBI, University of  Beira Interior,   Portugal; \    kaygorodov.ivan@gmail.com},
    Abror Khudoyberdiyev \footnote{V.I.Romanovskiy Institute of Mathematics Uzbekistan Academy of Sciences, Uzbekistan; National University of Uzbekistan,  Uzbekistan; \ khabror@mail.ru}  
\&
 Bauyrzhan Sartayev\footnote{Narxoz University,   Kazakhstan; \    baurjai@gmail.com}  
}

\end{center}

\noindent {\bf Abstract:}
{\it
We introduce the $\mathfrak M$arseille anti-construction $\mathfrak{Ma}$ for varieties of algebras defined by multilinear identities of degree $3$, obtained by changing the sign of the right-nested part of each defining polynomial.
This gives, for any variety $\Var$, a new variety $\mathfrak{Ma}$-$\Var$ that generalizes antiassociative algebras and the recently studied anti-Novikov and anti-right-symmetric $\big($pre-Jacobi--Jordan$\big)$ algebras.
We describe free objects in several $\mathfrak{Ma}$-varieties.
On the operadic level, over a field of characteristic $\mathrm{char}\,\Bbbk\neq 2$, we show $(\mathfrak{Ma}\text{-}\Var)^{(!)}=\mathfrak{Ma}\text{-}(\Var^{(!)})$ and that the Dong property is preserved by $\mathfrak{Ma}.$ Finally, we obtain the necessary and sufficient condition for $\mathfrak{Ma}\text{-}\Var=\Var$. 
The second part of the paper is dedicated to one another type of anti-algebras, called as $\mathcal{M}anin$ anti-algebras.
We give the description of such type of algebras via calculation of the white Manin product of suitable operads and 
present various examples of known $\mathcal{M}anin$ anti-algebras.
In the end, we proved that the symmetric affinization of an algebra gives a Jacobi-Jordan algebra if and only if the original algebra is a $\mathcal{M}a$-Novikov  algebra, that is an analog of a well-known results about affinization of Novikov algebras given by Balinskii and   Novikov. }

 \ 




\tableofcontents 
\newpage

\section*{Introduction}

The idea of producing new algebraic structures from a given one by changing $\big($or deriving$\big)$ the multiplication has a long history.
Already in the classical work of Albert \cite{albert}, such constructions appeared naturally in the study of non-associative algebras.
In its most general form, the principle was realized by Malcev \cite{malcev}: given the associative algebra ${\rm M}_n$ of $n\times n$ matrices over a field and a finite collection
$\Lambda=\{a_{ij},b_{ij},c_{ij}\}\subset {\rm M}_n$, one defines a new multiplication
\[
x\cdot_{\Lambda}y=\sum_{i,j}a_{ij}\,x\,b_{ij}\,y\,c_{ij},
\]
and obtains an algebra ${\rm M}_n^{(\Lambda)}$ such that every finite-dimensional algebra embeds into some ${\rm M}_n^{(\Lambda)}$.
Other important deformation procedures include isotopes and homotopes, which were intensively studied from various viewpoints 
$\big($see, for instance, \cite{alsaody,tk24}$\big)$.
A particular and very popular source of derived multiplications is given by the commutator and the anti-commutator:
\[
[x,y]=xy-yx,
\qquad
x\circ y=xy+yx,
\]
which lead to Lie-admissible and Jordan-admissible type conditions and their numerous variations 
$\big($cf.\ \cite{Alia0, Z17}$\big)$.
A conceptually different but extremely efficient framework for producing and comparing such ``derived'' objects is provided by operads.
In particular, for binary quadratic operads, one can systematically pass to the quadratic dual operad and study how algebraic identities transform under this duality.
This philosophy has proved useful in many structural problems, including the study of bases and dimensions of free objects and the construction of Gröbner--Shirshov bases for replicated (di-)algebras \cite{Kol2017}.

In recent years, several concrete ``anti''-versions of classical varieties were introduced and studied\footnote{It should be noted that the notion of Lie antialgebras was recently introduced by Ovsienko \cite{V11}. Although his definition is entirely different from the frameworks considered herein, the problem of constructing a unified approach for Ovsienko-type antialgebras remains an interesting and timely question.}:
antiassociative algebras appeared in the work of Okubo and Kamiya \cite{OK97} $\big($see also \cite{R26}$\big)$,
anti-pre-Lie and related anti-structures were developed in \cite{LB,LB24,CSL,N26,CL},
and anti-Novikov,  anti-Novikov--Poisson, 
anti-Poisson, and transposed anti-Poisson algebras were investigated in \cite{dN,dP}.
Further examples of anti-pre-Lie type structures on Lie algebras and their graded analogues were discussed in \cite{BD25,BD26}.
These developments indicate that the ``anti''-direction is not a single isolated phenomenon but a general mechanism that should be formulated uniformly.

\medskip 

The first aim  of the present paper is to introduce such a uniform mechanism for any variety of algebras defined by a family of multilinear identities of degree $3$.
Given a multilinear polynomial
\[
\mathfrak f=\sum_{\sigma\in\mathbb S_3}\alpha_\sigma\,\big(x_{\sigma(1)}x_{\sigma(2)}\big)x_{\sigma(3)}
+\sum_{\sigma\in\mathbb S_3}\beta_\sigma\,x_{\sigma(1)}\big(x_{\sigma(2)}x_{\sigma(3)}\big),
\]
we define its $\mathfrak M$arseille anti-polynomial by flipping the sign of the right-nested part.
Starting from a family $\Omega$ of defining degree $3$ identities of a variety $\Var$, we obtain the new variety $\mathfrak{Ma}$-$\Var$ defined by $\mathfrak{Ma}(\Omega)$.
This construction unifies, in a single operadic language, the known antiassociative case \cite{OK97}, the $(-1)$-Novikov and $(-1)$-right-symmetric cases \cite{dN}, and the anti-pre-Lie direction \cite{LB,CSL}.
We call the resulting algebras $\mathfrak M$arseille anti-algebras, since the initial idea emerged during a research stay at CIRM $\big($Marseille$\big)$.

The second goal of the paper is to study free algebras in several basic $\mathfrak{Ma}$-varieties.
Free objects are the main source of structural invariants (growth, nilpotency, dimensions of multilinear components, explicit bases), and they often reveal new combinatorics.
For example, the classical free bicommutative algebra has a well-understood monomial basis \cite{IsmDzhuma}, while Zinbiel-type identities are tightly connected to nilpotency phenomena \cite{dk05,T23}.
At the same time, anti- and ``transposed'' structures frequently produce unexpectedly small nilpotency indices.
One of our motivations is to make these computations explicit for $\mathfrak{Ma}$-versions and to compare them with known behavior in the original varieties.

On the operadic side, we also relate the $\mathfrak{Ma}$-construction to quadratic duality and to the Dong property.
The Dong property and the corresponding weight criterion were introduced and studied for binary quadratic operads in \cite{Dong}, and they lead to effective control of free derived algebras.
We show that $\mathfrak{Ma}$ behaves compatibly with quadratic duality and preserves the Dong property, which provides a conceptual explanation for several examples and produces new ones.
Finally, we connect $\mathfrak{Ma}$ with Malcev's decomposition of the space of invariant degree $3$ relations \cite{Malcev_x} and with 
the modern ``Malcev classification'' program for varieties such as associative and left-symmetric algebras \cite{MalcAs1,MalcAs2,MalcLS}.
This viewpoint clarifies when $\Var$ is stable under $\mathfrak{Ma}$ and explains why the passage $\Var\mapsto\mathfrak{Ma}\text{-}\Var$ naturally separates left- and right-nested components of identities.

\medskip

The last part of the paper is dedicated to one another type of anti-algebras, called as $\mathcal{M}anin$ anti-algebras.
We give the description of such type of algebras via calculation of the white Manin product of suitable operads and 
present various examples of known $\mathcal{M}anin$ anti-algebras.
In the end, we proved that the symmetric affinization of an algebra gives a Jacobi-Jordan algebra if and only if the original algebra is a $\mathcal{M}a$-Novikov  algebra, that is an analog of a well-known results about affinization of Novikov algebras gien by Balinskii and   Novikov in 1985. 

\medskip

\noindent{\bf Notations.}
Throughout the paper, we work over a field $\Bbbk$ of characteristic $0$, unless explicitly stated otherwise.
We do not repeat standard definitions 
$\big($Lie algebras, Leibniz algebras, solvable and nilpotent algebras, etc.$\big)$ 
and refer to the literature, see e.g.\ \cite{k23}.
For a set of vectors $S$ we denote by $\big\langle S\big\rangle$ the $\Bbbk$-linear span of $S$.
For the commutator, associator, anti-commutator, and anti-associator, we use
\[
[x,y]=xy-yx,
\qquad
(x,y,z)=(xy)z-x(yz),
\qquad
x\circ y=xy+yx,
\qquad
\{x,y,z\}=(xy)z+x(yz).
\]
When some identities are verified by symbolic computations, we use the computer algebra package {\rm Albert}.
We always assume that algebras under consideration are nontrivial 
$\big($the multiplication is not identically zero$\big)$.


\section{$\mathfrak M$arseille anti-algebras}

\begin{definition}
For a fixed multilinear polynomial 
\begin{center}
    $
\mathfrak f :=\sum\limits_{\sigma \in \mathbb{S}_3} \alpha_\sigma^j \big(x_{\sigma(1)}x_{\sigma(2)}\big)x_{\sigma(3)}+
    \sum\limits_{\sigma \in \mathbb{S}_3} \beta_\sigma^j x_{\sigma(1)}\big(x_{\sigma(2)}x_{\sigma(3)}\big),$
\end{center}   
    we define  that the $\mathfrak M$arseille anti-polynomial 
  $\mathfrak{Ma(f)}$ as follows 
\begin{center}  $
\mathfrak{Ma(f)} :=\sum\limits_{\sigma \in \mathbb{S}_3} \alpha_\sigma^j \big(x_{\sigma(1)}x_{\sigma(2)}\big)x_{\sigma(3)} - 
    \sum\limits_{\sigma \in \mathbb{S}_3} \beta_\sigma^j x_{\sigma(1)}\big(x_{\sigma(2)}x_{\sigma(3)}\big).$ 
\end{center}

\end{definition}

\begin{definition}
    If a variety of algebras $\rm{Var}$ is defined by the  family $\Omega = \{ f_i\}_{ 1\leq i \leq n}$ of $n$ polynomial identities of degree $3$, then the variety of $\mathfrak M$arseille anti-$\Omega$-algebras $\mathfrak{Ma}$-$\rm{Var}$ is defined by the following family of identities 
    $\mathfrak{Ma}(\Omega)= 
    \big\{ \mathfrak{Ma}(f_i)\big\}_{ 1\leq i \leq n}.$
\end{definition}

\noindent In the following, we give a list of $\mathfrak M$arseille anti-varieties for known varieties of algebras.

\begin{longtable}{|l|rcl|rcl|}

\hline
 & \multicolumn{3}{|c|}{$\rm{Var}$} & 
 \multicolumn{3}{|c|}{$\mathfrak{Ma}$-$\rm{Var}$}\\

\hline
\hline
Associative &$(xy)z $&$=$&$  x(yz)$ & 
$(xy)z $&$=$&$- \   x(yz)$ \\  
\hline
\hline
Shift associative &$(xy)z $&$=$&$  y(zx)$ & 
$(xy)z $&$=$&$- \   y(zx)$ \\  

\hline \hline
Leibniz &$(xy)z$&$=$&$ (xz)y +x(yz)  $  & 
$(xy)z$&$=$&$ (xz)y -x(yz)$ \\  

\hline\hline

Zinbiel & $(xy)z$&$=$&$ x (yz+zy)$ &
$(xy)z $&$=$&$- \  x (yz+zy)$ \\
\hline\hline

right symmetric  & 
$\big(x,y,z\big) $&$=$&$ \ \   \big(x,z,y\big) $ &
$\big\{x,y,z\big\} $&$=$&$ \ \   \big\{x,z,y\big\}$ \\

\hline\hline

right alternative  & 
$\big(x,y,z\big) $&$=$&$ -\big(x,z,y\big) $ &
$\big\{x,y,z\big\} $&$=$&$ -\big\{x,z,y\big\}$ \\

\hline
\hline

flexible   & 
$\big(x,y,z\big)  $&$=$&$    -  \big(z,y,x\big) $ &
$\big\{x,y,z \big\}  $&$=$&$ -    \big\{ z,y,x\big\} $ \\
\hline
\hline

antiflexible   & 
$\big(x,y,z\big)  $&$=$&$      \big(z,y,x\big) $ &
$\big\{x,y,z \big\}  $&$=$&$     \big\{ z,y,x\big\} $ \\
\hline
\hline

Lie-admissible    & 
\multicolumn{3}{l|}{$\sum\limits_{\sigma \in \mathbb S_3}\big(-1\big)^{\sigma} \big(x_{\sigma(1)},x_{\sigma(2)},x_{\sigma(3)} \big) =0$} &
\multicolumn{3}{l|}{$\sum\limits_{\sigma \in \mathbb S_3}\big(-1\big)^{\sigma} \big\{x_{\sigma(1)},x_{\sigma(2)},x_{\sigma(3)} \big\} =0$} \\
\hline
\hline

assosymmetric  & 
$\big(x,y,z\big)  $&$=$&$   \big(x,z,y\big) $ &
$\big\{x,y,z\big\} $&$=$&$  \big \{x,z,y\big\}$ \\

 & 
$\big(x,y,z\big)  $&$=$&$   \big(y,x,z\big) $ &
$\big\{x,y,z\big\}  $&$=$&$   \big\{y,x,z\big\} $ \\

\hline\hline

Novikov  & 
$\big(x,y,z\big) $&$=$&$ \big(x,z,y\big)  $ &
$\big\{x,y,z\big\}  $&$=$&$ \big\{ x,z,y\big\}$ \\

 & 
$x(yz)  $&$=$&$ y(xz)  $ &
$x(yz)  $&$=$&$ y(xz)  $ \\

\hline
\hline
alternative  & 
$\big(x,y,z\big)  $&$=$&$   -  \big(x,z,y\big)$ &
$\big\{x,y,z\big\}  $&$=$&$ - \big\{x,z,y\big\}$ \\

 & 
$\big(x,y,z\big)  $&$=$&$   -  \big(y,x,z\big)  $ &
$\big\{x,y,z\big\}  $&$=$&$ -  \big\{ y,x,z\big\}  $ \\

 \hline
\hline

$\big(-1,1\big)$- & 
$\big(x,y,z\big)  $&$=$&$  -  \big(x,z,y\big) $ &
$\big\{x,y,z\big\}  $&$=$&$  -\   \big\{x,z,y\big\}$ \\
 & 
\multicolumn{3}{l|}{$\big(x,y,z\big)+\big(y,z,x\big)+\big(z,x,y\big)  \ = \ 0  $} &
\multicolumn{3}{l|}{$ \big\{x,y,z\big\}+\big\{y,z,x\big\}+\big\{z,x,y\big\} \ = \ 0 $} \\

\hline

\end{longtable}

\noindent Remarks:
\begin{enumerate}
    \item[$(\rm{A})$] 
The notion of $\mathfrak{Ma}$-Novikov and $\mathfrak{Ma}$-right symmetric algebras  coincides with the notion of $\big(-1\big)$-Novikov and 
$\big(-1\big)$-right symmetric algebras introduced in {\rm \cite{dN}}.

\item[$(\rm{B})$] The notion of     $\mathfrak{Ma}$-antiflexible  algebras coincides with the notion of 
$\mathcal A$flexible  algebras introduced in {\rm \cite{BBK}}.

\item[$(\rm{C})$] The notion of     $\mathfrak{Ma}$-Zinbiel algebras generalizes the notion of 
anti-Zinbiel algebras introduced in {\rm \cite{LB23}}.

\item[$(\rm{D})$] The notion of   $\mathfrak{Ma}$-Leibniz and  $\mathfrak{Ma}$-Zinbiel algebras are special particular cases of generalized Leibniz and generalized Zinbiel algebras introduced in {\rm \cite{MO14}}.

\item[$(\rm{E})$] The notion of $\mathfrak{Ma}$-right symmetric algebras coincides with the notion of pre-Jacobi-Jordan  algebras introduced in {\rm \cite{BBMM}}.

\item[$(\rm{F})$] The polarization of the $\mathfrak{Ma}$-Lie-admissible operad coincides with the operad $\Alia$ $\big($see  \cite{Alia0,Alia}$\big).$ The polarized operation satisfies the identity
\[
[a,b] \circ  c + [b,c]\circ a + [c,a] \circ b \ =\ 0.
\]


\end{enumerate}

Let us denote by ${\rm \mathfrak{Ma}\textrm{-}Var}$ and ${\rm \mathfrak{Ma}\textrm{-}Var\<X\>}$ the variety of $\mathfrak{Ma}$-Var algebras and the corresponding free algebra from the variety ${\rm \mathfrak{Ma}\textrm{-}Var}$, respectively.

\subsection{$\mathfrak{Ma}$-associative and $\mathfrak{Ma}$-alternative algebras}
    \begin{proposition}[see \cite{R26}]
Each $\mathfrak{Ma}$-associative algebra is nilpotent of index $4$.
\end{proposition}

\begin{proposition}\label{basisAAs}
Let $\mathfrak{Ma}\textrm{-}\As\<X\>$ be the free $\mathfrak{Ma}$-associative algebra, generated by the set $X=\big\{x_i \big\}_{i \in I}.$ 
Then the following set of elements 
\begin{center}
    ${\mathfrak B}=\big\{  \ x_i,\  x_{i}x_{j},\  x_{i} (x_{j}x_{k})  \ \big\}_{i, j, k \in I}$
\end{center}
gives a basis of $\mathfrak{Ma}$-$\As\<X\>.$
    \end{proposition}
\begin{proof}
Let us remember that $\mathfrak{Ma}\textrm{-}\As\<X\>$ is a $3$-step nilpotent algebra, 
hence, $x_{k_1}\ldots x_{k_t}=0$ for $t \ge 4.$
All basis elements are constructed from words of length $1$, $2$, or $3.$
Due to anti-associativity, all words given by 
$\big\{\ x_{i} (x_{j}x_{k}) \  \big\}_{i,   j, k \in I} $ are independent.
Hence, ${\mathfrak B}$ is a basis of  
$\mathfrak{Ma}\textrm{-}\As\<X\>.$ \end{proof}

Direct computations via the algebraic software Albert give 
the following result (about shift associative algebras, see \cite{shift}):
 
\begin{proposition} 
Each $\mathfrak{Ma}$-shift associative algebra is nilpotent of index $5.$
\end{proposition}

\begin{proposition}\label{basisAAs}
Let $\mathfrak{Ma}\textrm{-}S\As\<X\>$ be the free $\mathfrak{Ma}$-shift associative algebra, generated by the set $X=\big\{x_i \big\}_{i \in I}.$ 
Then the following set of elements 
\begin{center}
    ${\mathfrak B}=\big\{  \ x_i,\  x_{i}x_{j},\  x_{i} (x_{j}x_{k}), \  (x_{i} x_{j})(x_{m} x_{k})   \ \big\}_{i, j, k, m \in I; \  j<m}$
\end{center}
gives a basis of $\mathfrak{Ma}\textrm{-}S\As\<X\>.$
    \end{proposition}
\begin{proof}
Let us remember that $\mathfrak{Ma}\textrm{-}S\As\<X\>$ is a $4$-step nilpotent algebra, 
hence, $x_{k_1}\ldots x_{k_t}=0$ for $t \ge 5.$
All basis elements are constructed from words of length $1$, $2$,  $3,$ or $4.$
Due to anti-shift associativity, all words given by 
$\big\{\ x_{i} (x_{j}x_{k}) \  \big\}_{i,   j, k \in I} $ are independent. It is easy to see that 
\begin{longtable}{lclclcl}
$(xy)(zt)$&$=$&$-\  y\big((zt)x\big) $&$=$&$ y \big(t(xz)\big);$ \\ 
$(xy)(zt)$&$=$&$-\ \big(t(xy)\big)z$&$= $&$\big((yt)x\big)z$&$=$&$- \ (xz)(yt).$
\end{longtable}
Hence, ${\mathfrak B}$ is a basis of   $\mathfrak{Ma}\textrm{-}S\As\<X\>.$ \end{proof}

Analogically, we use algebraic software Albert to obtain the following result:
\begin{proposition}
An algebra $\mathfrak{Ma}$-$\Alt\<X\>$ is nilpotent of index $6$.
\end{proposition}

Let $\Var$ be a variety of algebras. Denote by $\Var_2$ the binary part of $\Var$, that is, the class of all algebras $A$ such that every subalgebra of $A$ generated by two elements belongs to $\Var$.
The defined identities of varieties of binary Leibniz $\big($see \cite{Leib_2}$\big)$, 
binary Zinbiel $\big($see \cite{Zinb_2}$\big)$, and 
 binary perm $\big($see \cite{binaryperm}$\big)$ algebras are described by 
  Dzhumadildaev, Ismailov, Mashurov, Smadyarov, Kunanbayev, and Sartayev.

It is well known that the variety of alternative algebras coincides with the variety of binary associative algebras. In contrast, the analogous statement fails for the pair of varieties of $\mathfrak{Ma}$-alternative and $\mathfrak{Ma}$-associative algebras.

\begin{theorem}
For the varieties of $\mathfrak{Ma}$-alternative and $\mathfrak{Ma}$-associative algebras, one has
\[
\mathfrak{Ma}\textrm{-}\Alt \neq \mathfrak{Ma}\textrm{-}\As_2 .
\]
\end{theorem}

\begin{proof}
It is immediate that
\[
\mathfrak{Ma}\textrm{-}\Alt \supseteq \mathfrak{Ma}\textrm{-}\As_2 .
\]
To show that the inclusion is strict, we compare the free two-generated algebras
$\mathfrak{Ma}\textrm{-}\Alt\langle a,b\rangle$ and $\mathfrak{Ma}\textrm{-}\As\langle a,b\rangle$.
Using the computer algebra system Albert, we compute bases of these algebras.

For $\mathfrak{Ma}\textrm{-}\Alt\langle a,b\rangle$ we obtain the following basis elements:
\[
a,\ b,\ bb,\ ba,\ ab,\ aa,
\]
\[
(bb)b,\ (bb)a,\ (ba)b,\ (ab)b,\ (ba)a,\ (ab)a,\ (aa)b,\ (aa)a,
\]
\[
\big((ba)b\big)b,\ \big((ab)a\big)b,\ \big((aa)b\big)b,\ \big((aa)a\big)b.
\]
Moreover, starting from degree $5$ the algebra $\mathfrak{Ma}\textrm{-}\Alt\langle a,b\rangle$ is nilpotent.
On the other hand, by Proposition~\ref{basisAAs}, the algebra
$\mathfrak{Ma}\textrm{-}\As\langle a,b\rangle$ is nilpotent of index $4$.

Therefore,
\begin{center}
$\mathfrak{Ma}\textrm{-}\Alt\<a,b\>\ncong \mathfrak{Ma}\textrm{-}\As_2\<a,b\>$
\ and\ 
$\mathfrak{Ma}\textrm{-}\Alt \varsupsetneq \mathfrak{Ma}\textrm{-}\As_2.$
\end{center}
\end{proof}

\subsection{$\mathfrak{Ma}$-Leibniz algebras}

Let us study a free algebra from the variety of $\mathfrak{Ma}$-Leibniz algebras. 

\begin{lemma}
An algebra $\mathfrak{Ma}$-${\rm Leib\<X\>}$ satisfies the following identities:
\begin{equation}\label{ALeib1}
a\big(bc+cb\big) \ =\ 0,
\end{equation}
\begin{equation}\label{ALeib2}
a\big(b(cd)\big)\ =\ 0,
\end{equation}
\begin{equation}\label{ALeib3}
a\big((bc)d\big) \ =\ 0.
\end{equation}
\end{lemma}
\begin{proof}
It can be proved using the computer algebra software program Albert. 
\end{proof}

Any non-associative monomial can be represented as
\[
\big(\cdots((x_iw_1)w_2)\cdots \big)w_n,
\]
where $x_i$ is a generator and $w_i$ is also non-associative monomial. On the set of non-associative monomials of $\mathfrak{Ma}$-${\rm Leib\<X\>}$, we defined degree-lexicographical order. We say that a monomial
\[
m\ =\ \big(\cdots((x_iw_1)w_2)\cdots\big)w_n
\]
belongs to the set $\mathcal{AL}$ if it satisfies the following conditions:
\begin{itemize}
    \item $w_1\leq w_2\leq \ldots\leq w_n$;
    \item ${\rm deg}(w_i)\leq 2$;
    \item for every $i$ such that $w_i=x_{i_1}x_{i_2}$, we have $x_{i_1}\leq x_{i_2}.$
\end{itemize}

\begin{example}\label{ExAleib}
Let us list the set $\mathcal{AL}$ in the multilinear case up to degree $8$. 
Denote by $\mathcal{AL}_i$ the subset of $\mathcal{AL}$ consisting of monomials of degree $i$. 
In degree $3$, we have
\[
\mathcal{AL}_3\ =\ \big\{\,x_{1}(x_{2}x_{3}),\ x_{2}(x_{1}x_{3}),\ x_{3}(x_{1}x_{2}),\ (x_{1}x_{2})x_{3},\ (x_{2}x_{1})x_{3},\ (x_{3}x_{1})x_{2}\,\big\}.
\]
For higher degrees, we indicate only the bracketing type and the number of monomials satisfying the required conditions. If $i=4$, then
\begin{longtable}{lcccr}
$\big((x_{i_1}x_{i_2})x_{i_3}\big)x_{i_4}$&$\rightarrow$&$ {\rm C}_{4}^{1}$&$=$&$4,$\\
$(x_{i_1}x_{i_2})(x_{i_3}x_{i_4})$&$\rightarrow$&${\rm C}_{4}^{2}\cdot 2$&$=$&$12.$
\end{longtable}
If $i=5$, then
\begin{longtable}{lcccr}
$\big(((x_{i_1}x_{i_2})x_{i_3})x_{i_4}\big)x_{i_5}$&$\rightarrow$&${\rm C}_{5}^{1}$&$=$&$5,$\\
$\big((x_{i_1}x_{i_2})x_{i_3}\big)(x_{i_4}x_{i_5})$&$\rightarrow$&${\rm C}_{5}^{3}\cdot 3$&$=$&$30,$\\ 
$\big(x_{i_1}(x_{i_2}x_{i_3})\big)(x_{i_4}x_{i_5})$&$\rightarrow$&${\rm C}_{5}^{3}\cdot \frac  32$&$=$&$15.$
\end{longtable}
If $i=6$, then
\begin{longtable}{lcccr}
$\big((((x_{i_1}x_{i_2})x_{i_3})x_{i_4})x_{i_5}\big)x_{i_6}$&$\rightarrow$&${\rm C}_{6}^{1}$&$=$&$6,$\\
$\Big(\big((x_{i_1}x_{i_2})x_{i_3}\big)x_{i_4}\Big)(x_{i_5}x_{i_6})$&$\rightarrow$&${\rm C}_{6}^{2}\cdot 4$&$=$&$60,$\\ 
$\big((x_{i_1}x_{i_2})(x_{i_3}x_{i_4})\big)(x_{i_5}x_{i_6})$&$\rightarrow$&${\rm C}_{6}^{4}{\rm C}_{4}^{2}\cdot \frac{2}{2}$&$=$&$90.$
\end{longtable}
If $i=7$, then
\begin{longtable}{lcccr}
$\Big(\big((((x_{i_1}x_{i_2})x_{i_3})x_{i_4})x_{i_5}\big)x_{i_6}\Big)x_{i_7}$&$\rightarrow$&${\rm C}_{7}^{1}$&$=$&$7,$\\
$\Big(\big(((x_{i_1}x_{i_2})x_{i_3})x_{i_4}\big)x_{i_5}\Big)(x_{i_6}x_{i_7})$&$\rightarrow$&${\rm C}_{7}^{2}\cdot 5$&$=$&$105,$\\
$\Big(\big((x_{i_1}x_{i_2})x_{i_3}\big)(x_{i_4}x_{i_5})\Big)(x_{i_6}x_{i_7})$&$\rightarrow$&${\rm C}_{7}^{2}{\rm C}_{5}^{2}\cdot \frac 32$&$=$&$315,$\\
$\Big(\big(x_{i_1}(x_{i_2}x_{i_3})\big)(x_{i_4}x_{i_5})\Big)(x_{i_6}x_{i_7})$&$\rightarrow$&${\rm C}_{7}^{2}{\rm C}_{5}^{2}{\rm C}_{3}^{2} \cdot \frac 1{3!}$&$=$&$105.$
\end{longtable}
If $i=8$, then
\begin{longtable}{lcccr}
$\Big(\big(((((x_{i_1}x_{i_2})x_{i_3})x_{i_4})x_{i_5})x_{i_6}\big)x_{i_7}\Big)x_{i_8} $&$\rightarrow $&${\rm C}^1_8$&$=$&$8,$\\
$\Big(\big((((x_{i_1}x_{i_2})x_{i_3})x_{i_4})x_{i_5}\big)x_{i_6}\Big)(x_{i_7}x_{i_8}) $&$\rightarrow $&${\rm C}_{8}^{2} \cdot 6 $&$= $&$168,$\\
$\Big(\big(((x_{i_1}x_{i_2})x_{i_3})x_{i_4}\big)(x_{i_5}x_{i_6})\Big)(x_{i_7}x_{i_8}) $&$\rightarrow $&${\rm C}_{8}^{2}{\rm C}_{6}^{2}\cdot \frac 4 2 $&$= $&$840,$\\
$\Big(\big((x_{i_1}x_{i_2})(x_{i_3}x_{i_4})\big)(x_{i_5}x_{i_6})\Big)(x_{i_7}x_{i_8}) $&$\rightarrow $&${\rm C}_{8}^{2}{\rm C}_{6}^{2}{\rm C}_{4}^{2}\cdot \frac 2{3!} $&$= $&$840.$
\end{longtable}
Similarly, we constructed $\mathcal{AL}_9$ and $\mathcal{AL}_{10}$, and obtained
\[
|\mathcal{AL}_9|=6876
\qquad\text{and}\qquad
|\mathcal{AL}_{10}|=26200.
\]
\end{example}

\begin{theorem}\label{BasisALeib}
The set $\mathcal{AL}$ is a basis of the algebra $\mathfrak{Ma}$-${\rm Leib\<X\>}$.
\end{theorem}
\begin{proof}
Firstly, we show that the set $\mathcal{AL}$ spans $\mathfrak{Ma}$-${\rm Leib}\<X\>$. 
Let $m$ be an arbitrary monomial of degree $k$ of the form
\[
m\ =\ \big(\cdots((x_i w_1)w_2)\cdots\big)w_n.
\]
Using the identities \eqref{ALeib2} and \eqref{ALeib3}, we can reduce $m$ whenever $\deg(w_j)\ge 3$. Using the identity \eqref{ALeib1}, we order the generators inside $w_j$. 
Next, we can interchange $w_j$ and $w_{j+1}$ by the defining identity of $\mathfrak{Ma}$-${\rm Leib}\<X\>$, namely,
\[
(aw_j)w_{j+1}=a(w_{j+1}w_{j})+(aw_{j+1})w_{j}.
\]
If both $w_j$ and $w_{j+1}$ are generators and $w_j>w_{j+1}$, then the monomials
$a(w_{j+1}w_{j})$ and $(aw_{j+1})w_{j}$ belong to $\mathcal{AL}$. 
If $\deg(w_{j})=\deg(w_{j+1})=2$, then the identity above reduces to
\[
(aw_j)w_{j+1}=(aw_{j+1})w_{j}.
\]
Thus, the set of remaining monomials coincides with the set $\mathcal{AL}$.

Let us consider the algebra $\mathcal{A}\<X\>$ with a basis $\mathcal{AL}$. 
We define a multiplication on $\mathcal{A}\<X\>$ as follows. 
Let $A_i\in\mathcal{AL}$ be an arbitrary basis monomial of degree $i$.

\begin{itemize}
\item If $A_t\in\mathcal{AL}$ has degree $t\ge 3$, then
\[
A_k\cdot A_t=0
\qquad (k\ge 1).
\]

\item If $t=2$, and
\[
A_k\ =\ \big(\cdots((x_i w_1)w_2)\cdots\big)w_n,
\]
then for any degree-$2$ basis monomial $w_p$ we set
\[
A_k\cdot w_p
=\big(\cdots((((\cdots((x_i w_1)w_2)\cdots)w_{r-1})w_p)w_r)\cdots\big)w_n,
\]
where $w_{r-1}\le w_p\le w_r$.

\item If $t=1$, and
\[
A_k=\big(\cdots(((\cdots((x_i x_{t_1})x_{t_2})\cdots)x_{t_l})w_{l+1})\cdots\big)w_n,
\]
then for any generator $x_p$ we define
\begin{multline*}
A_k\cdot x_p
=\big(\cdots(((\cdots((((\cdots((x_i x_{t_1})x_{t_2})\cdots)x_{t_{r-1}})x_p)x_{t_r})\cdots)x_{t_l})w_{l+1})\cdots\big)w_n \\
+\sum_{s=r}^{l}\big(\cdots((((\cdots(((\cdots((x_i x_{t_1})x_{t_2})\cdots)x_{t_{r-1}})x_{t_r})\cdots)x_{t_l})(x_p x_{t_s}))w_{l+1})\cdots\big)w_n,
\end{multline*}
where $x_{t_{r-1}}\le x_p<x_{t_r}$, and in the final sum there is no generator $x_{t_s}$ among $x_{t_r},\dots,x_{t_l}$.
\end{itemize}
A direct computation shows that the algebra $\mathcal{A}\<X\>$ belongs to $\mathfrak{Ma}\text{-}$${\rm Leib}$, that is, it satisfies the defining identity of $\mathfrak{Ma}\text{-}$${\rm Leib}$. 
It remains to note that
\[
\mathcal{A}\<X\>\cong \mathfrak{Ma}\text{-}{\rm Leib}\<X\>.
\]
\end{proof}

As verified in Example~\ref{ExAleib}, the first values of $\dim\big({\rm \mathfrak{Ma}\text{-}Leib}(n)\big)$ are given by
\begin{center}
\begin{tabular}{c|cccccccccc}
$n$ & 1 & 2 & 3 & 4 & 5 & 6 & 7 & 8 & 9 & 10 \\
\hline
${\rm \mathfrak{Ma}\text{-}Leib}(n)$ & 1 & 2 & 6 & 16 & 50 & 156 & 532 & 1856 & 6876 & 26200
\end{tabular}
\end{center}
This sequence coincides with OEIS~A013989, which counts the total number of adjacent transpositions in all involutions (self-inverse permutations) of the set $\{1,2,\dots,n+2\}$. 
In particular, if we set $a(n)={\rm \mathfrak{Ma}\text{-}Leib}(n)$, then up to degree 10, we have the recurrence
\[
a(n)=n\left(\frac{a(n-1)}{n-1}+a(n-2)\right),
\qquad a(1)=1,\quad a(2)=2.
\]

\subsection{$\mathfrak{Ma}$-Zinbiel algebras}

The present section is dedicated to a more detailed study of $\mathfrak{Ma}$-Zinbiel and $\mathfrak{Ma}$-alternative algebras. 
Let us remember that each Zinbiel algebra is right-commutative. The following statement gives an $\mathfrak{Ma}$-analog for ${\rm anti}$-Zinbiel algebras.

\begin{proposition}\label{arz}
    Each $\mathfrak{Ma}$-Zinbiel algebra is a right-commutative algebra.
\end{proposition}
\begin{proof}
    We have 
    $(xy)z \ =\ - \ x(zy+yz) \ = \ -\ x (yz+zy) \ = \   (xz)y.$
\end{proof}

Indeed, there is no need to construct the basis of $\mathfrak{Ma}$-$\Zinb\<X\>$, since using the algebraic software Albert, we obtain
\begin{proposition}
An algebra $\mathfrak{Ma}$-$\Zinb\<X\>$ is nilpotent of index $5$.
\end{proposition}

In 1972, Kantor introduced $\big($left$\big)$ conservative algebras as a generalization of Jordan algebras \cite{K72}
 (also, see surveys on recent studies of conservative algebras and superalgebras \cite{k23}). Namely, a vector space ${\rm V}$ with a multiplication $\cdot$ is called a $\big($left$\big)$ {\it conservative algebra} if there is a new multiplication $*:{\rm V}\times {\rm V}\rightarrow {\rm V}$ such that
\begin{center}
$\big[L_b^\cdot, [L_a^\cdot,\cdot]\big]\ =\ -\big[L_{a*b}^\cdot,\cdot\big],$
\ where \ 
$\big[L_c^\cdot, {\rm F}\big] (x,y)\ =\ c\cdot{\rm F}(x, y)-{\rm F}(c\cdot x,  y)-{\rm F}(x, c\cdot y).$
\end{center}
In other words, the following identity holds for all $a,b,x,y\in {\rm V}$:
\begin{multline*}
b\big(a(xy)-(ax)y-x(ay)\big)-
a\big((bx)y\big)+\big(a(bx)\big)y+(bx)(ay)\\
-a\big(x(by)\big)+(ax)(by)+x\big(a(by)\big)
+(a *b)(xy)-\big((a* b)x\big)y-x\big((a *b)y\big) \ =\ 0.
\end{multline*}

A result from \cite{KLP15} states that each (left) Zinbiel algebra is conservative. Below, we have a similar statement for ${\rm anti}$-Zinbiel algebras.

\begin{proposition}
Each $\big($left$\big)$ $\mathfrak{Ma}$-Zinbiel algebra is a $\big($left$\big)$ conservative algebra  with the  additional multiplication given by $a*b\ =\ -\ (ab+ba).$
\end{proposition}
\begin{proof}
Thanks to Proposition \ref{arz},
each left $\mathfrak{Ma}$-Zinbiel algebra is left-commutative.
Thanks to \cite{KLP15}, each left-conservative algebra
that satisfies a quasi-identity 
\begin{center}
    $a(bx)\  =\  {\rm F}(a, b)x$
\end{center} for a new multiplication $F$ is a conservative. 
Hence, our left $\mathfrak{Ma}$-Zinbiel algebra is conservative with the additional multiplication  ${\rm F}(a,b)\ =\ -\ (ab+ba).$
 \end{proof}

\subsection{$\mathfrak{Ma}$-Novikov algebras}

Now we consider the algebra $\mathfrak{Ma}\text{-}$${\rm Nov}\<X\>$ and describe several properties.

\begin{lemma}
The algebra $\mathfrak{Ma}\text{-}$${\rm Nov}\<X\>$ satisfies the following identities:
\begin{align}
(ad)(bc)&=a\big((bd)c\big)=b\big((ad)c\big), \label{ANovId1}\\
\big(a(cd)\big)b&=a\big(c(db)\big), \label{ANovId2}\\
(db)\big((ce)a\big)&=c\big(((de)b)a\big), \label{ANovId3}\\
c\big(((de)b)a\big)&=d\big(((cb)e)a\big), \label{ANovId4}
\end{align}
\end{lemma}

\begin{proof}
These identities were verified by computer algebra program Albert.
\end{proof}

We also need a more general relation of \eqref{ANovId4}.

\begin{lemma}\label{ANovGenId}
For any $n\ge 5$, the following identity holds in $\mathfrak{Ma}\text{-}$${\rm Nov}\<X\>$:
\[
x_{i_1}\bigl(((\cdots((x_{i_{2}}x_{i_{3}})x_{i_{4}})\cdots)x_{i_{n-1}})x_{i_n}\bigr)
=
x_{i_2}\bigl(((\cdots((x_{i_{1}}x_{i_{n-1}})x_{i_{n-2}})\cdots)x_{i_{3}})x_{i_n}\bigr).
\]
\end{lemma}

\begin{proof}
Using the left-commutative identity and \eqref{ANovId1}, we obtain
\begin{multline*}
x_{i_1}\bigl(((\cdots((x_{i_{2}}x_{i_{3}})x_{i_{4}})\cdots)x_{i_{n-1}})x_{i_n}\bigr)
=
\big((\cdots((x_{i_{2}}x_{i_{3}})x_{i_{4}})\cdots)x_{i_{n-1}}\big)(x_{i_1}x_{i_n})
\\
=
\big((\cdots((x_{i_{2}}x_{i_{3}})x_{i_{4}})\cdots)x_{i_{n-2}}\big)\bigl((x_{i_1}x_{i_{n-1}})x_{i_n}\bigr).
\end{multline*}
Now we apply \eqref{ANovId1} repeatedly:
\begin{multline*}
\big((\cdots((x_{i_{2}}x_{i_{3}})x_{i_{4}})\cdots)x_{i_{n-2}}\big)\bigl((x_{i_1}x_{i_{n-1}})x_{i_n}\bigr)
\\
=
\big((\cdots((x_{i_{2}}x_{i_{3}})x_{i_{4}})\cdots)x_{i_{n-3}}\big)\bigl(((x_{i_1}x_{i_{n-1}})x_{i_{n-2}})x_{i_n}\bigr)
=\ldots
\\
=
x_{i_2}\bigl(((\cdots((x_{i_{1}}x_{i_{n-1}})x_{i_{n-2}})\cdots)x_{i_{3}})x_{i_n}\bigr).
\end{multline*}
\end{proof}

\begin{theorem}\label{ThmANovBasis}
The basis of $\mathfrak{Ma}\text{-}$${\rm Nov}\<X\>$ coincides with the basis of the free bicommutative algebra ${\rm BiCom}\<X\>$.
\end{theorem}

Before proving Theorem~\ref{ThmANovBasis}, let us recall the basis of ${\rm BiCom}\<X\>$ given in \cite{IsmDzhuma}.
A monomial $m$ belongs to the basis of ${\rm BiCom}\<X\>$ if and only if it has the form
\[
m=x_{i_1}\bigl(x_{i_2}(\cdots(x_{i_k}((\cdots((x_{i_{k+1}}x_{i_{k+2}})x_{i_{k+3}})\cdots)x_{i_n}))\cdots)\bigr),
\]
where
\[
i_1\le i_2\le \cdots \le i_{k+1}
\qquad\text{and}\qquad
i_{k+2}\ge i_{k+3}\ge \cdots \ge i_n.
\]
Denote this set of monomials by $\mathcal{AN}$.

\begin{proof}[Proof of Theorem~\ref{ThmANovBasis}]
First, we show that every monomial in $\mathfrak{Ma}\text{-}$${\rm Nov}\<X\>$ can be rewritten as a linear combination of monomials from $\mathcal{AN}$.
We prove this by induction on the degree.

The base cases are $n=3$ and $n=4$. For $n=3$, we have
\[
x_{2}(x_{1}x_{3})=x_{1}(x_{2}x_{3}),\qquad
x_{3}(x_{2}x_{1})=x_{2}(x_{3}x_{1}),\qquad
x_{3}(x_{1}x_{2})=x_{1}(x_{3}x_{2}),
\]
and
\[
(x_{1}x_{2})x_{3}=(x_{1}x_{3})x_{2}+x_{1}(x_{3}x_{2})-x_{1}(x_{2}x_{3}),
\]
\[
(x_{2}x_{1})x_{3}=(x_{2}x_{3})x_{1}+x_{2}(x_{3}x_{1})-x_{1}(x_{2}x_{3}),
\]
\[
(x_{3}x_{1})x_{2}=(x_{3}x_{2})x_{1}+x_{2}(x_{3}x_{1})-x_{1}(x_{3}x_{2}).
\]

We endow $\mathfrak{Ma}\text{-}$${\rm Nov}\<X\>$ with the degree--lexicographical order.
We also write $m\equiv m'$ if $m$ and $m'$ have the same bracketing type and
$m-m'$ is a linear combination of basis monomials strictly smaller than $m$ in this order.
In degree $3$, this yields
\[
(x_{1}x_{2})x_{3}\equiv (x_{1}x_{3})x_{2},\qquad
(x_{2}x_{1})x_{3}\equiv (x_{2}x_{3})x_{1},\qquad
(x_{3}x_{1})x_{2}\equiv (x_{3}x_{2})x_{1}.
\]
As example, the identity \eqref{ANovId3} shows that
\[
(db)\big((ce)a\big)\equiv 0.
\]

For $n=4$, by the left-commutative identity, every monomial of the form
$x_{i_1}\bigl(x_{i_2}(x_{i_3}x_{i_4})\bigr)$ can be rewritten as one of the following monomials:
\[
x_{1}\bigl(x_{2}(x_{3}x_{4})\bigr),\quad
x_{1}\bigl(x_{2}(x_{4}x_{3})\bigr),\quad
x_{1}\bigl(x_{3}(x_{4}x_{2})\bigr),\quad
x_{2}\bigl(x_{3}(x_{4}x_{1})\bigr).
\]

For a monomial of the form $x_{i_1}\bigl((x_{i_2}x_{i_3})x_{i_4}\bigr)$, using the inductive hypothesis and \eqref{ANovId1}, we obtain
\[
x_{i_1}\bigl((x_{i_2}x_{i_3})x_{i_4}\bigr)\equiv
x_{i_1}\bigl((x_{i_2}x_{i_4})x_{i_3}\bigr)
=
x_{i_2}\bigl((x_{i_1}x_{i_4})x_{i_3}\bigr).
\]
In particular, modulo $\equiv$, we can permute the generators within the pairs $(x_{i_1},x_{i_2})$ and $(x_{i_3},x_{i_4})$ in order to achieve the required ordering.

Next, consider a monomial of the form $(x_{i_1}x_{i_2})(x_{i_3}x_{i_4})$.
By the left-commutative identity and the inductive hypothesis, we have
\[
(x_{i_1}x_{i_2})(x_{i_3}x_{i_4})\equiv 0,
\]
that is, there is no basis monomial with the same bracketing type, and this monomial can be expressed as a linear combination of smaller basis monomials.

Similarly, by \eqref{ANovId2} and the previous cases, we have
\[
\bigl(x_{i_1}(x_{i_2}x_{i_3})\bigr)x_{i_4}\equiv 0.
\]

Finally, using the defining identities of $\mathfrak{Ma}\text{-}$${\rm Nov}$ together with the previous reductions, we obtain
\[
\bigl((x_{i_1}x_{i_2})x_{i_3}\bigr)x_{i_4}
\equiv
\bigl((x_{i_1}x_{i_2})x_{i_4}\bigr)x_{i_3}
\equiv
\bigl((x_{i_1}x_{i_4})x_{i_2}\bigr)x_{i_3}.
\]
This completes the verification of the base case $n=4$.

Let $n\ge 5$. We consider products $A_pA_{n-p}$, where $A_k$ denotes an arbitrary monomial of degree $k$.

\smallskip
\noindent
\textbf{Case 1: $A_1A_{n-1}$.}
By the inductive hypothesis, every monomial of the form $A_1A_{n-1}$ can be written as
\[
x_{i_1}\Bigl(x_{i_2}\bigl(\cdots\bigl(x_{i_k}((\cdots((x_{i_{k+1}}x_{i_{k+2}})x_{i_{k+3}})\cdots)x_{i_n})\bigr)\cdots\bigr)\Bigr).
\]
Applying Lemma \ref{ANovGenId}, we obtain
\begin{multline*}
x_{i_1}\Bigl(x_{i_2}\bigl(\cdots\bigl(x_{i_k}((\cdots((x_{i_{k+1}}x_{i_{k+2}})x_{i_{k+3}})\cdots)x_{i_n})\bigr)\cdots\bigr)\Bigr)
\equiv\\
x_{i_1}\Bigl(x_{i_2}\bigl(\cdots\bigl(x_{i_{k+1}}((\cdots((x_{i_k}x_{i_{k+2}})x_{i_{k+3}})\cdots)x_{i_n})\bigr)\cdots\bigr)\Bigr).
\end{multline*}
Hence, using the left-commutative identity and Lemma \ref{ANovGenId}, we can reorder (modulo $\equiv$) the generators
\[
x_{i_1},x_{i_2},\ldots,x_{i_{k+1}}
\]
to achieve the required nondecreasing order.
Moreover, by the inductive hypothesis, the generators
\[
x_{i_{k+2}},x_{i_{k+3}},\ldots,x_{i_n}
\]
can also be ordered (modulo $\equiv$) in the t

\smallskip
\noindent
\textbf{Case 2: $A_tA_{n-t}$ with $t>1$ and $n-t>1$.}
By \eqref{ANovId1}, we have
\[
A_tA_{n-t}\equiv 0,
\]
that is, such monomials reduce to a linear combination of smaller basis monomials.

\smallskip
\noindent
\textbf{Case 3: $A_{n-1}A_1$.}
By \eqref{ANovId2}, every monomial in $A_{n-1}A_1$ vanishes modulo $\equiv$, except the fully left-normed monomial
\[
\big(\cdots((x_{i_1}x_{i_2})x_{i_3})\cdots\big)x_{i_n}.
\]
For a monomial of this form, by the previous reductions, we can permute the indices $i_2,\ldots,i_n$ modulo $\equiv$; more precisely,
\[
\big(\cdots((x_{i_1}x_{i_2})x_{i_3})\cdots\big)x_{i_n}\equiv
\big(\cdots((x_{i_1}x_{i_{\sigma(2)}})x_{i_{\sigma(3)}})\cdots\big)x_{i_{\sigma(n)}},
\qquad \sigma\in \mathbb{S}_{\{2,3,\ldots,n\}}.
\]
This completes the inductive step.

Let us consider the algebra $\mathcal{A}\langle X\rangle$ with basis $\mathcal{AN}$.
We define a bilinear multiplication on $\mathcal{A}\langle X\rangle$ on basis elements as follows.

Let $x_{i_{k+2}},\ldots,x_{i_{n}},y_{j_{l+2}},\ldots,y_{j_{m-1}}\geq y_{j_m}$,
\[
u=
x_{i_1}\bigl(x_{i_2}(\cdots(x_{i_k}L_x)\cdots)\bigr),
\qquad
L_x=\big(\cdots((x_{i_{k+1}}x_{i_{k+2}})x_{i_{k+3}})\cdots\big)x_{i_n},
\]
and
\[
v=
y_{j_1}\bigl(y_{j_2}(\cdots(y_{j_l}L_y)\cdots)\bigr),
\qquad
L_y=\big(\cdots((y_{j_{l+1}}y_{j_{l+2}})y_{j_{l+3}})\cdots)y_{j_m}.
\]
Set
\[
uv=
z_{p_1}\bigl(z_{p_2}(\cdots(z_{p_{k+l+1}}L_z)\cdots)\bigr),
\qquad
L_z=((\cdots((z_{p_{k+l+2}}z_{p_{k+l+3}})z_{p_{k+l+4}})\cdots)z_{p_{n+m-1}}\big)y_{j_m},
\]
where the indices are determined by the following conditions:
\[
\{x_{i_1},\ldots,x_{i_{k+1}},\,y_{j_1},\ldots,y_{j_{l+1}}\}
=
\{z_{p_1},\ldots,z_{p_{k+l+2}}\},
\]
\[
\{x_{i_{k+2}},\ldots,x_{i_n},\,y_{j_{l+2}},\ldots,y_{j_{m-1}}\}
=
\{z_{p_{k+l+3}},\ldots,z_{p_{n+m-1}}\},
\]
and the variables are ordered so that
\[
z_{p_1}\le \cdots \le z_{p_{k+l+2}},
\qquad
z_{p_{k+l+3}}\ge \cdots \ge z_{p_{n+m-1}}.
\]
Equivalently, the first block is written in nondecreasing order, while the second block is written in nonincreasing order.

The only remaining case is when, in the tail $L_z$, the element $y_{j_m}$ occurs inside the nonincreasing block, namely,
\[
z_{p_{k+l+3}}\ge \cdots \ge z_{p_r}\ge y_{j_m} \ge z_{p_{r+1}}\ge \cdots \ge z_{p_{n+m-1}}.
\]
Without loss of generality, we assume that $y_{j_m} \ge z_{p_{k+l+3}}\ge \cdots \ge z_{p_{n+m-1}}.$
As a result, $L_z$ can be written as the following alternating sum:
\begin{multline*}
L_z=\big((\cdots((z_{p_{k+l+2}}z_{p_{k+l+3}})z_{p_{k+l+4}})\cdots)z_{p_{n+m-1}}\big)y_{j_m}\\
=\big(\cdots(((z_{p_{k+l+2}}y_{j_m})z_{p_{k+l+3}})z_{p_{k+l+4}})\cdots\big)z_{p_{n+m-1}}\\
\quad +\,z_{p_{k+l+2}}\bigl((\cdots((y_{j_m}z_{p_{k+l+3}})z_{p_{k+l+4}})\cdots)z_{p_{n+m-1}}\bigr)
-\,z_{p_{k+l+2}}\bigl((\cdots((z_{p_{n+m-1}}y_{j_m})z_{p_{k+l+3}})\cdots)z_{p_{n+m-2}}\bigr)\\
\quad -\,z_{p_{k+l+2}}\Bigl(z_{p_{n+m-1}}\bigl((\cdots((y_{j_m}z_{p_{k+l+3}})z_{p_{k+l+4}})\cdots)z_{p_{n+m-2}}\bigr)\Bigr)\\
\quad +\,z_{p_{k+l+2}}\Bigl(z_{p_{n+m-1}}\bigl((\cdots((z_{p_{n+m-2}}y_{j_m})z_{p_{k+l+3}})\cdots)z_{p_{n+m-3}}\bigr)\Bigr)
\;+\;\cdots\\
\quad +\,(-1)^n z_{p_{k+l+2}}\Bigl(z_{p_{n+m-1}}\bigl(z_{p_{n+m-2}}(\cdots(y_{j_m}z_{p_{k+l+3}})\cdots)\bigr)\Bigr)\\
\quad -\,(-1)^n z_{p_{k+l+2}}\Bigl(z_{p_{n+m-1}}\bigl(z_{p_{n+m-2}}(\cdots(z_{p_{k+l+3}}y_{j_m})\cdots)\bigr)\Bigr).
\end{multline*}

We illustrate this expansion with several low-degree examples:
\[
\big((ac)b\big)d=
\big((ad)c\big)b+a\bigl((dc)b\bigr)-a\bigl((bd)c\bigr)-a\bigl(b(dc)\bigr)+a\bigl(b(cd)\bigr),
\]
\begin{multline*}
\big(((ad)c)b\big)e
=\big(((ae)d)c\big)b+a\bigl(((ed)c)b\bigr)-a\bigl(((be)d)c\bigr)-a\bigl(b((ed)c)\bigr)+a\bigl(b((ce)d)\bigr)\\
\qquad\qquad
+a\bigl(b(c(ed))\bigr)-a\bigl(b(c(de))\bigr),
\end{multline*}
and
\begin{multline*}
((((ae)d)c)b)f
=((((af)e)d)c)b+a\bigl((((fe)d)c)b\bigr)-a\bigl((((bf)e)d)c\bigr)-a\bigl(b(((fe)d)c)\bigr)\\
+a\bigl(b(((cf)e)d)\bigr)
+a\bigl(b(c((fe)d))\bigr)-a\bigl(b(c((df)e))\bigr)
-a\bigl(b(c(d(fe)))\bigr)+a\bigl(b(c(d(ef)))\bigr).
\end{multline*}

A direct computation shows that the algebra $\mathcal{A}\<X\>$ belongs to $\mathfrak{Ma}\text{-}$${\rm Nov}$, that is, it satisfies the defining identity of $\mathfrak{Ma}\text{-}$${\rm Nov}$. 
It remains to note that
\[
\mathcal{A}\<X\>\cong \mathfrak{Ma}\text{-}{\rm Nov}\<X\>.
\]

\end{proof}

For the operad $\mathfrak{Ma}\text{-}$${\rm Nov}$, we have
\[
\dim\bigl(\mathfrak{Ma}\text{-}{\rm Nov}(n)\bigr)=2^{n}-2.
\]

For the pair $\Var\langle X\rangle$ and its dialgebra analogue $\di\text{-}\Var\langle X\rangle$, there is a well-known relationship between their dimensions and bases established in \cite{Kol2017}. Namely, if
\[
\dim(\Var(n))=k_n,
\]
then
\[
\dim(\di\text{-}\Var(n))=n\cdot k_n.
\]
Moreover, a basis of $\Var\langle X\rangle$ yields, in an explicit and canonical way, a basis of $\di\text{-}\Var\langle X\rangle$.

In contrast, as illustrated by the algebras $\mathfrak{Ma}\text{-}$$\As\langle X\rangle$, $\mathfrak{Ma}\text{-}$$\Leib\langle X\rangle$, and $\mathfrak{Ma}\text{-}$$\Nov\langle X\rangle$, no analogous correspondence appears to hold between $\Var\langle X\rangle$ and its $\mathfrak{Ma}\text{-}$version $\Var\langle X\rangle$. In general, neither the dimensions nor the monomial bases of these two free objects are directly related.

\subsection{Koszul dual of $\mathfrak{Ma}\text{-}\Var$ and the Dong property}

Let us compute the dual operad for $\mathfrak{Ma}\text{-}$$\Var$. The Lie-admissibility condition for $S\otimes U$ gives the defining identities of the operad $\mathfrak{Ma}\text{-}$$\Var^{(!)}$, where $S$ is an algebra from $\mathfrak{Ma}\text{-}$$\Var$. Firstly, let us consider the case $\Var=\Leib$. So, by Theorem \ref{BasisALeib}, we have
\begin{multline*}
\big[[a\otimes u,b\otimes v],c\otimes w\big]=(ab)c\otimes (uv)w-(ba)c\otimes (vu)w-
c(ab)\otimes w(uv)+c(ba)\otimes w(vu)=\\
(ab)c\otimes (uv)w-(ba)c\otimes (vu)w-
c(ab)\otimes w(uv)-c(ab)\otimes w(vu),
\end{multline*}
\begin{multline*}
\big[[b\otimes v,c\otimes w],a\otimes u\big]=(bc)a\otimes (vw)u-(cb)a\otimes (wv)u-
a(bc)\otimes u(vw)+a(cb)\otimes u(wv)=\\
((ba)c+b(ac))\otimes (vw)u-((ca)b+c(ab))\otimes (wv)u-
a(bc)\otimes u(vw)-a(bc)\otimes u(wv)
\end{multline*}
and
\begin{multline*}
\big[[c\otimes w,a\otimes u],b\otimes v\big]=(ca)b\otimes (wu)v-(ac)b\otimes (uw)v-
b(ca)\otimes v(wu)+b(ac)\otimes v(uw)=\\
(ca)b\otimes (wu)v-((ab)c+a(bc))\otimes (uw)v+
b(ac)\otimes v(wu)+b(ac)\otimes v(uw).
\end{multline*}
The sum of the same elements on the left side of the tensors, we obtain the following result: 
\begin{proposition}
The operad $\mathfrak{Ma}\text{-}$$\Leib^{(!)}$ is an operad which satisfies the following identities:
\[
(uv)w=(uw)v\;\;\text{and}\;\; (vw)u+v(wu)+v(uw)=0.
\]
It is not difficult to see that $\mathfrak{Ma}\text{-}\Leib^{(!)}=\mathfrak{Ma}\text{-}\Zinb$.
\end{proposition}

If $\Var=\text{anti-}\Com\text{-}\mathfrak{Ma}\text{-}\As$, then 
\begin{multline*}
\big[[a\otimes u,b\otimes v],c\otimes w\big]=(ab)c\otimes (uv)w-(ba)c\otimes (vu)w-
c(ab)\otimes w(uv)+c(ba)\otimes w(vu)=\\
-a(bc)\otimes (uv)w-a(bc)\otimes (vu)w-
a(bc)\otimes w(uv)-a(bc)\otimes w(vu),
\end{multline*}
\begin{multline*}
\big[[b\otimes v,c\otimes w],a\otimes u\big]=(bc)a\otimes (vw)u-(cb)a\otimes (wv)u-
a(bc)\otimes u(vw)+a(cb)\otimes u(wv)=\\
-a(bc)\otimes (vw)u-a(bc)\otimes (wv)u-
a(bc)\otimes u(vw)-a(bc)\otimes u(wv)
\end{multline*}
and
\begin{multline*}
\big[[c\otimes w,a\otimes u],b\otimes v\big]=(ca)b\otimes (wu)v-(ac)b\otimes (uw)v-
b(ca)\otimes v(wu)+b(ac)\otimes v(uw)=\\
-a(bc)\otimes (wu)v-a(bc)\otimes (uw)v-
a(bc)\otimes v(wu)-a(bc)\otimes v(uw).
\end{multline*}
The sum of the same elements on the left side of the tensors, we obtain the expected following result:
\begin{proposition}[see, \cite{Z17}]
The operad $\text{anti-}\Com\text{-}\mathfrak{Ma}\text{-}\As^{(!)}$ is Jacobi-Jordan.
\end{proposition}
The analogical calculations give the following result:
\begin{proposition}[see, \cite{dN}]
The operad $\mathfrak{Ma}\text{-}\Nov$ is self-dual.
\end{proposition}
More general result:

\begin{theorem}
Let $\Var$ be a quadratic operad over a field $\Bbbk$ with
$\operatorname{char}\Bbbk\neq 2$. Then
\[
(\mathfrak{Ma}\text{-}\Var)^{(!)}\ =\ 
\mathfrak{Ma}\text{-}(\Var^{(!)}).
\]
\end{theorem}

\begin{proof}
Let $\Var=\mathcal F(E)/(R)$ be a quadratic presentation, where $\mathcal F(E)$ is the free symmetric operad
generated by a binary $\mathbb S$-module $E$ and $R\subset \mathcal F(E)(3)$ is the $\mathbb S_3$-submodule of
quadratic relations. We use the definition of the quadratic dual operad via orthogonal complement:
\[
\Var^{(!)}=\mathcal F(E^{\vee})/(R^{\perp}),
\qquad
R^{\perp}:=\{\,g\in \mathcal F(E^{\vee})(3)\mid \langle r,g\rangle=0\ \text{for all } r\in R\,\},
\]
where $\langle\cdot,\cdot\rangle$ is the standard nondegenerate $\mathbb S_3$-invariant pairing
between $\mathcal F(E)(3)$ and $\mathcal F(E^{\vee})(3)$.

In arity $3$, we have the direct sum decomposition
\[
\mathcal F(E)(3)=\mathcal F_L(E)(3)\oplus \mathcal F_R(E)(3),
\]
where $\mathcal F_L(E)(3)$ is spanned by the monomials
$\bigl(x_{\sigma(1)}x_{\sigma(2)}\bigr)x_{\sigma(3)}$ and
$\mathcal F_R(E)(3)$ is spanned by the monomials
$x_{\sigma(1)}\bigl(x_{\sigma(2)}x_{\sigma(3)}\bigr)$, for $\sigma\in\mathbb S_3$.
Hence every $f\in \mathcal F(E)(3)$ is uniquely written as $f=f_L+f_R$ with
$f_L\in \mathcal F_L(E)(3)$ and $f_R\in \mathcal F_R(E)(3)$.

By definition,
\[
\mathfrak{Ma}(f):=f_L-f_R.
\]
In particular, $\mathfrak{Ma}$ is $\Bbbk$-linear, $\mathbb S_3$-equivariant, and
\begin{equation}\label{eq:MaProperties}
\mathfrak{Ma}^2=\mathrm{id},\qquad
\mathfrak{Ma}\big|_{\mathcal F_L(E)(3)}=\mathrm{id},\qquad
\mathfrak{Ma}\big|_{\mathcal F_R(E)(3)}=-\,\mathrm{id}.
\end{equation}
Moreover, since $\operatorname{char}\Bbbk\neq 2$, we have the projections
\begin{equation*}
f_L=\frac12\bigl(f+\mathfrak{Ma}(f)\bigr),
\qquad
f_R=\frac12\bigl(f-\mathfrak{Ma}(f)\bigr).
\end{equation*}

We define $\mathfrak{Ma}$ on $\mathcal F(E^{\vee})(3)$ by the same rule:
it acts by $+1$ on the left-nested component and by $-1$ on the right-nested component.
Then for all $f\in\mathcal F(E)(3)$ and $g\in\mathcal F(E^{\vee})(3)$, we have
\begin{equation}\label{eq:MaAdjoint}
\langle \mathfrak{Ma}(f),\,g\rangle=\langle f,\,\mathfrak{Ma}(g)\rangle.
\end{equation}
Indeed, writing $f=f_L+f_R$ and $g=g_L+g_R$, we obtain
\[
\langle \mathfrak{Ma}(f),g\rangle
=\langle f_L-f_R,\,g_L+g_R\rangle
=\langle f_L,g_L\rangle-\langle f_R,g_R\rangle,
\]
because the mixed pairings $\langle f_L,g_R\rangle$ and $\langle f_R,g_L\rangle$ vanish.
On the other hand,
\[
\langle f,\mathfrak{Ma}(g)\rangle
=\langle f_L+f_R,\,g_L-g_R\rangle
=\langle f_L,g_L\rangle-\langle f_R,g_R\rangle,
\]
which proves \eqref{eq:MaAdjoint}.

We claim that
\begin{equation}\label{eq:OrthCommutesMa}
\bigl(\mathfrak{Ma}(R)\bigr)^{\perp}=\mathfrak{Ma}\bigl(R^{\perp}\bigr).
\end{equation}
Let $g\in \big(\mathfrak{Ma}(R)\big)^{\perp}$. Then for every $r\in R$ we have
\[
\langle \mathfrak{Ma}(r),\,g\rangle=0.
\]
Using \eqref{eq:MaAdjoint}, this is equivalent to $\langle r,\,\mathfrak{Ma}(g)\rangle=0$ for all $r\in R$,
hence $\mathfrak{Ma}(g)\in R^{\perp}$. Applying \eqref{eq:MaProperties}, we get
$g=\mathfrak{Ma}\big(\mathfrak{Ma}(g)\big)\in \mathfrak{Ma}(R^{\perp})$.
Thus $\big(\mathfrak{Ma}(R)\big)^{\perp}\subseteq \mathfrak{Ma}(R^{\perp})$.

Conversely, let $g\in \mathfrak{Ma}(R^{\perp})$. Then $g=\mathfrak{Ma}(h)$ for some $h\in R^{\perp}$.
For any $r\in R$ we obtain, using \eqref{eq:MaAdjoint} and \eqref{eq:MaProperties},
\[
\langle \mathfrak{Ma}(r),\,g\rangle
=\langle \mathfrak{Ma}(r),\,\mathfrak{Ma}(h)\rangle
=\langle r,\,h\rangle
=0.
\]
Hence $g\in \big(\mathfrak{Ma}(R)\big)^{\perp}$, and therefore
$\mathfrak{Ma}(R^{\perp})\subseteq \big(\mathfrak{Ma}(R)\big)^{\perp}$.
This proves \eqref{eq:OrthCommutesMa}.
\end{proof}

Let us clarify the obtained result. By definition,
\[
\mathfrak{Ma}\text{-}\Var=\mathcal F(E)/\big(\mathfrak{Ma}(R)\big).
\]
Taking the quadratic dual, we obtain
\[
(\mathfrak{Ma}\text{-}\Var)^{(!)}
=\mathcal F(E^{\vee})/\bigl(\mathfrak{Ma}(R)\bigr)^{\perp}.
\]
Using \eqref{eq:OrthCommutesMa}, this is
\[
(\mathfrak{Ma}\text{-}\Var)^{(!)}
=\mathcal F(E^{\vee})/\mathfrak{Ma}(R^{\perp}).
\]
On the other hand, since $\Var^{(!)}=\mathcal F(E^{\vee})/(R^{\perp})$, applying the construction $\mathfrak{Ma}$
means replacing $R^{\perp}$ by $\mathfrak{Ma}(R^{\perp})$, hence
$ 
\mathfrak{Ma}\text{-}(\Var^{(!)})=\mathcal F(E^{\vee})/\mathfrak{Ma}(R^{\perp}).$
Therefore,
\[
(\mathfrak{Ma}\text{-}\Var)^{(!)}=\mathfrak{Ma}\text{-}(\Var^{(!)}),
\]
as required.

In \cite{Dong}, the criterion is given for a quadratic operad to satisfy the Dong property. For $\mathfrak{Ma}\text{-}\Var$, we have 
\begin{theorem}
The operad $\mathfrak{Ma}\text{-}\Var$ satisfies the Dong property if and only if $\Var$ satisfies the Dong property.
\end{theorem}
\begin{proof}
We use the criterion from \cite{Dong}. Namely, an operad $\Var$ satisfies the Dong property if and only if $\Var$ fulfills the identities
\begin{equation}\label{eq:ToInside}
(x_1\circ_j x_2)\circ_i x_3
=\sum_{k,l\in I}
\alpha^{i,j}_{k,l}\,(x_3\circ_k x_2)\circ_l x_1
+\beta^{i,j}_{k,l}\,(x_1\circ_k x_3)\circ_l x_2,
\end{equation}
where, for each $t\in\{i,j,k,l\}$, the operation $\circ_t$ is either the operation $\cdot$ of $\Var$ or its opposite operation, that is,
$\circ_t=\cdot$ or $\circ_t=\cdot^{op}$.
Thus, if $\Var$ satisfies \eqref{eq:ToInside}, then the operad $\mathfrak{Ma}\text{-}\Var$ satisfies the analogous identity, where each monomial involving
$\circ_i=\cdot^{op}$ and $\circ_l=\cdot^{op}$ appears with the opposite sign.
Nevertheless, the operad $\mathfrak{Ma}\text{-}\Var$ still satisfies \eqref{eq:ToInside}. 
\end{proof}

\begin{corollary}
Among the algebras listed in the table above, the operads $\mathfrak{Ma}\text{-}$associative, $\mathfrak{Ma}\text{-}$shift-associative, $\mathfrak{Ma}\text{-}$assosymmetric, $\mathfrak{Ma}\text{-}$Novikov, and $\mathfrak{Ma}\text{-}$alternative are precisely those that satisfy the Dong property.
\end{corollary}

\subsection{Malcev classification and $\mathfrak{Ma}\text{-}$algebras}

As shown in \cite{Malcev_x}, the space of invariant relations of degree $3$ decomposes as
$\mathfrak{L}=\mathfrak{L}_1+\mathfrak{L}_2,$
where $\mathfrak{L}_1:=\mathfrak{A}_1+\mathfrak{B}_1+\mathfrak{C}_1+\mathfrak{D}_1$, and the subspaces are defined by the following polynomial relations:
\begin{gather}
\mathfrak{A}_1:\;(ab)c+(ba)c+(ac)b+(ca)b+(bc)a+(cb)a=0,
\label{as1} \\
\mathfrak{B}_1:\;(ab)c-(ba)c-(ac)b+(ca)b+(bc)a-(cb)a=0,
\label{as2}
\end{gather}
\begin{gather}\label{as3}
\mathfrak{C}_1:\;
\begin{cases}
(ab)c+(ba)c-(bc)a-(cb)a=0, \\
(ac)b+(ca)b-(cb)a-(bc)a=0,
\end{cases}
\end{gather}
\begin{gather}\label{as4}
\mathfrak{D}_1:\;
\begin{cases}
(ab)c+(ac)b-(cb)a-(ca)b=0, \\
(ab)c+(ac)b-(bc)a-(ba)c=0.
\end{cases}
\end{gather}
Analogously, we set $\mathfrak{L}_2:=\mathfrak{A}_2+\mathfrak{B}_2+\mathfrak{C}_2+\mathfrak{D}_2$, where
$\mathfrak{A}_2,\mathfrak{B}_2,\mathfrak{C}_2,\mathfrak{D}_2$ are defined by the same relations written with right-normed bracketing. It is worth noting that for $\mathfrak{Z}\in\{\mathfrak{A},\mathfrak{B},\mathfrak{C},\mathfrak{D}\}$ and $i\in\{1,2\}$, the algebra which is defined by the identity $\mathfrak{Z}_i$ is stable under the $\mathfrak{Ma}$-construction, that is, $
\mathfrak{Ma}\text{-}\mathfrak{Z}_i=\mathfrak{Z}_i.$
Also, one may consider the operad $\Var$ endowed with an additional identity of type $\mathfrak{Z}_i$. All operads obtained in this way are called the \emph{Malcev classification} of the operad $\Var$. For instance, the Malcev classification of the associative operad was studied in \cite{MalcAs1, MalcAs2}. An analogous study for the variety of left-symmetric (pre-Lie) algebras was carried out in \cite{MalcLS}.
Moreover, every multilinear identity of degree $3$ can be expressed as a linear combination of \eqref{as1}, \eqref{as2}, \eqref{as3}, \eqref{as4} and their right-normed analogues. Therefore, we obtain the following commutative diagram:
\begin{center}
\setlength{\unitlength}{1pt}
\begin{picture}(400,90)
  \put(125,75){$\Var$}
  \put(240,75){$\Var_{\mathfrak{Z}_i}$}
  \put(135,70){\vector(0,-1){35}}
  \put(250,70){\vector(0,-1){35}}
  \put(115,20){$\mathfrak{Ma}\text{-}\Var$}
  \put(240,20){$\mathfrak{Ma}\text{-}\Var_{\mathfrak{Z}_i}$}
  \put(160,25){\vector(1,0){75}}
  \put(160,79){\vector(1,0){75}}
\end{picture}
\end{center}
\vspace*{-\baselineskip}
\vspace*{-\baselineskip}

\begin{theorem}
Let $\Var$ be a quadratic operad. Then
\[
\Var=\mathfrak{Ma}\text{-}\Var
\]
if and only if the defining identities of $\Var$ can be chosen in the form
\begin{equation}\label{MaVar-left}
\sum_{\sigma \in \mathbb{S}_3} \alpha_\sigma \bigl(x_{\sigma(1)}x_{\sigma(2)}\bigr)x_{\sigma(3)}=0
\end{equation}
and
\begin{equation}\label{MaVar-right}
\sum_{\sigma \in \mathbb{S}_3} \beta_\sigma \,x_{\sigma(1)}\bigl(x_{\sigma(2)}x_{\sigma(3)}\bigr)=0.
\end{equation}
\end{theorem}
\begin{proof}
The implication $(\Leftarrow)$ follows immediately from the definition.
Indeed, if the defining identities of $\Var$ can be chosen in the form \eqref{MaVar-left} and \eqref{MaVar-right}, then applying $\mathfrak{Ma}$ fixes every identity of the form \eqref{MaVar-left} and changes the sign of every identity of the form \eqref{MaVar-right}. Hence the same family of identities defines both $\Var$ and $\mathfrak{Ma}\text{-}\Var$, and therefore $\Var=\mathfrak{Ma}\text{-}\Var$.

\smallskip
To prove $(\Rightarrow)$, assume that $\Var=\mathfrak{Ma}\text{-}\Var$. Let $f$ be any defining multilinear identity of degree $3$ for $\Var$. Then $f$ has the form
\[
f=\sum_{\sigma \in \mathbb{S}_3} \alpha_\sigma \bigl(x_{\sigma(1)}x_{\sigma(2)}\bigr)x_{\sigma(3)}
+\sum_{\sigma \in \mathbb{S}_3} \beta_\sigma \,x_{\sigma(1)}\bigl(x_{\sigma(2)}x_{\sigma(3)}\bigr)=0.
\]
Since $\Var=\mathfrak{Ma}\text{-}\Var$, the operad $\Var$ also satisfies the identity $\mathfrak{Ma}(f)=0$, that is,
\[
\sum_{\sigma \in \mathbb{S}_3} \alpha_\sigma \bigl(x_{\sigma(1)}x_{\sigma(2)}\bigr)x_{\sigma(3)}
-\sum_{\sigma \in \mathbb{S}_3} \beta_\sigma \,x_{\sigma(1)}\bigl(x_{\sigma(2)}x_{\sigma(3)}\bigr)=0.
\]
Assuming $\operatorname{char}\Bbbk\neq 2$ (in particular, in characteristic $0$), adding these two identities we obtain
\[
\sum_{\sigma \in \mathbb{S}_3} \alpha_\sigma \bigl(x_{\sigma(1)}x_{\sigma(2)}\bigr)x_{\sigma(3)}=0,
\]
and subtracting the second identity from the first one we obtain
\[
\sum_{\sigma \in \mathbb{S}_3} \beta_\sigma \,x_{\sigma(1)}\bigl(x_{\sigma(2)}x_{\sigma(3)}\bigr)=0.
\]
Thus, the defining identities of $\Var$ can be chosen separately in the left-bracketed form \eqref{MaVar-left} and the right-bracketed form \eqref{MaVar-right}, as required.
\end{proof}
By Theorem and the observation above, we see that Malcev’s classification and the
$\mathfrak{Ma}$-construction for a quadratic operad $\Var$ are closely related. In particular,
the passage $\Var\mapsto \mathfrak{Ma}\text{-}\Var$ separates the defining quadratic relations into
their left- and right-nested parts, which is precisely the kind of structural information that is
reflected in Malcev-type decompositions of identities. Therefore, it is natural to study the operad
$\mathfrak{Ma}\text{-}\Var$ through the lens of Malcev’s classification and to compare the resulting
classes with those arising from $\Var$ itself.


\section{$\mathcal{M}$anin anti-algebras}

\subsection{(Anti)commutative algebras with an identity of degree $3$}

\begin{proposition}
    Let ${\rm A}$ be a commutative algebra satisfying a nontrivial identity of degree $3,$ then ${\rm A}$ is associative or Jacobi--Jordan.
\end{proposition}

\begin{proof}
It is easy to see, that there are $\beta,\gamma \in  \mathbb{C},$ such that 
${\rm A}$ satisfies the following identity 
\begin{center}
    $(xy)z+\beta (yz)x +\gamma (zx)y=0.$\end{center}
Which is equivalent to identities 
\begin{center}    
$(yz)x+\beta (zx)y +\gamma (xy)z=0$ and 
$(zx)y+\beta (xy)z +\gamma (yz)x=0.$  
\end{center}
Summarizing the last $3$ identities, we have 
$(1+\beta+\gamma) \big((xy)z+   (yz)x +   (zx)y\big)=0.$  Then,
\begin{enumerate}
    \item[I.] If $\beta+\gamma \neq-1,$ then ${\rm A}$ is a Jacobi-Jordan algebra.

    \item[II.] If $\gamma=-\beta-1,$ then ${\rm A}$  satisfies the following pair of identities
    \begin{center}
$(xy)z+\beta (yz)x-(1+\beta)(xz)y=0$ and 
$(xz)y+\beta (zy)x -(1+\beta)(xy)z=0;$\\
$(xy)z+\beta (yz)x -(1+\beta)(xz)y=0$ and 
$(yx)z+\beta (xz)y -(1+\beta)(yz)x=0.$
    \end{center}
\noindent Subsequently, subtracting these identities, we have 
\begin{center}
    $(2+\beta)\big((xy)z-(xz)y\big)=0$\  and \  
    $(2\beta+1)\big((xy)z-(xz)y\big)=0.$
\end{center}
\noindent     It means that  ${\rm A}$ is associative.

\end{enumerate}

\end{proof}

\begin{proposition}
    Let ${\rm A}$ be an anticommutative algebra satisfying a nontrivial identity of degree $3,$ then ${\rm A}$ is anti-associative or Lie.
\end{proposition}

\begin{proof}
It is easy to see, that there are $\beta,\gamma \in  \mathbb{C},$ such that 
${\rm A}$ satisfies the following identity 
\begin{center}
    $(xy)z+\beta (yz)x +\gamma (zx)y=0.$\end{center}
Which is equivalent to identities 
\begin{center}    
$(yz)x+\beta (zx)y +\gamma (xy)z=0$ and 
$(zx)y+\beta (xy)z +\gamma (yz)x=0.$  
\end{center}
Summarizing the last $3$ identities, we have 
$(1+\beta+\gamma) \big((xy)z+   (yz)x +   (zx)y\big)=0.$  Then,
\begin{enumerate}
    \item[I.] If $\beta+\gamma \neq-1,$ then ${\rm A}$ is a Lie algebra.

    \item[II.] If $\gamma=-\beta-1,$ then ${\rm A}$  satisfies the following  identities
    \begin{center}
$(xy)z+\beta (yz)x+(1+\beta)(xz)y=0,$ 
 
$(xz)y+\beta (zy)x +(1+\beta)(xy)z=0,$

$(yx)z+\beta (xz)y +(1+\beta)(yz)x=0.$
    \end{center}
\noindent  Summing the first identity with the second and third, we have
\begin{center}
    $(2+\beta)\big((xy)z+(xz)y\big)=0$\  and \  
    $(2\beta+1)\big((xz)y+(yz)x\big)=0.$
\end{center}
\noindent     It means that  ${\rm A}$ is anti-associative.
\end{enumerate}
\end{proof}

Let us fix the following notations for   operads that govern the  varieties of  algebras:
\begin{longtable}{lcl}
${\rm AC}$ &$-$&   the variety of associative commutative algebras;\\  
${\rm AA}$ &$-$&    the variety of anti-associative anti-commutative algebras;\\  
${\rm Lie}$ &$-$&    the variety of Lie  algebras;\\  
${\rm JJ}$ &$-$&    the variety of Jacobi--Jordan  algebras;\\ 
${\rm AM}$ &$-$&    the variety of anti-commutative algebras;\\  
${\rm CM}$ &$-$&   the variety of   commutative algebras.  

\end{longtable}

\begin{proposition}
    The set $\Big\{ {\rm AC},  \ {\rm AA},\  {\rm Lie},\  {\rm JJ},  {\rm AM},  {\rm CM} \Big\}$ with the white Manin product gives  a Clifford semigroup.
    In particular, the multiplication is given by the following table

\begin{longtable}{c|c|c|c|c|c|c}

$\circ$ & {\rm AC} & {\rm AA} & {\rm Lie} & {\rm JJ} & {\rm CM} & {\rm AM} \\
\hline 
{\rm AC} & {\rm AC} & {\rm AA} & {\rm Lie} & {\rm JJ} &  {\rm CM} & {\rm AM} \\
\hline 

{\rm AA} & {\rm AA} & {\rm AC} & {\rm JJ} & {\rm Lie} & {\rm AM} & {\rm CM} \\
\hline 

{\rm Lie} & {\rm Lie} & {\rm JJ} & {\rm CM} & {\rm AM} & {\rm AM} & {\rm CM} \\
\hline 

{\rm JJ} & {\rm JJ} & {\rm Lie} & {\rm AM} & {\rm CM} & {\rm CM} & {\rm AM}\\
\hline
{\rm CM} & {\rm CM} & {\rm AM} & {\rm AM} & {\rm CM} & {\rm CM} & {\rm AM}\\
\hline
{\rm AM} & {\rm AM} & {\rm CM} & {\rm CM} & {\rm AM} & {\rm AM} & {\rm CM}\\

\hline

\end{longtable}

\end{proposition}

\subsection{$\mathcal{M}$anin anti-algebras}

\begin{definition}
Let $\mathcal{O}$ be the operad that governs the variety of $\Omega$-algebras, 
where $\Omega$ is a family of polynomial identities.  
Then ${\rm AA} \circ \mathcal{O}$ be the operad that governs the variety of $\mathcal{M}$anin anti-$\Omega$-algebras 
$\big($or $\mathcal{M}$a-$\Omega$-algebras$\big)$.

\end{definition}

\begin{theorem}
    If a variety of algebras is defined by the  family $\Omega$ of $n$ polynomial identities
\begin{center}
    $\Omega= \Big\{\sum\limits_{\sigma \in \mathbb{S}_3}
    \big(\alpha_\sigma^j (x_{\sigma(1)}x_{\sigma(2)})x_{\sigma(3)}+
    \beta_\sigma^j x_{\sigma(1)}(x_{\sigma(2)}x_{\sigma(3)}) \big) \ = \ 0 \Big\}_{1\leq j\leq n},$
\end{center}
then the variety of Manin anti-$\Omega$-algebras is defined by the following family of identities 

\begin{center}
    $\Omega^{\mathfrak a}= \Big\{\sum\limits_{\sigma \in \mathbb{S}_3}  (-1)^{\sigma}\big(\alpha_\sigma^j (x_{\sigma(1)}x_{\sigma(2)})x_{\sigma(3)}-
     \beta_\sigma^j x_{\sigma(1)}(x_{\sigma(2)}x_{\sigma(3)}) \big)\ = \ 0  \Big\}_{1\leq j\leq n}.$
\end{center}

\end{theorem}

\begin{proof}
This result may be checked with a straightforward computation. Namely, firstly we analyze the following monomials in a free antiassociative anticommutative algebra:
\begin{longtable}{lcllcllcllcl}
$(ab)c$ &$=$&$ (ab)c,$ &
$(ac)b$ &$=$&$ -(ab)c,$ &
$a(bc)$ &$=$&$ -(ab)c,$ &
$a(cb)$ &$=$&$ (ab)c,$ \\
$(bc)a$ &$=$&$ (ab)c,$ &
$(ba)c$ &$=$&$ -(ab)c,$ &
$b(ca)$ &$=$&$ -(ab)c,$ &
$b(ac)$ &$=$&$ (ab)c,$ \\
$(ca)b$ &$=$&$ (ab)c,$ &
$(cb)a$ &$=$&$ -(ab)c,$ &
$c(ab)$ &$=$&$ -(ab)c,$ &
$c(ba)$ &$=$&$ (ab)c,$ 
\end{longtable}
\noindent i.e.,   $(x_1x_2)x_3  =  (-1)^\sigma (x_{\sigma(1)}x_{\sigma(2)})x_{\sigma(3)}$
and $(x_1x_2)x_3  =  -  (-1)^\sigma x_{\sigma(1)}(x_{\sigma(2)}x_{\sigma(3)})$ for each $\sigma \in \mathbb S_3.$

The operad ${\rm AA} \circ \mathcal{O}$ is generated by $4$-dimensional space $({\rm AA} \circ \mathcal{O})(2)$
spanned by
\begin{longtable}{lcllcl}
$(x_1 \prec x_2) $&$= $&$x_1x_2 \otimes x_1 \bullet x_2,$&$(x_2 \prec x_1)$&$ = $&$x_2x_1 \otimes x_2 \bullet x_1,$\\
$(x_1 \succ x_2)$&$ = $&$x_2x_1 \otimes x_1 \bullet x_2,$&$ (x_2 \succ x_1)$&$ =$&$ x_1x_2 \otimes x_2 \bullet x_1,$
\end{longtable}
\noindent where 
$(x_1 \prec x_2) \ = \ x_1x_2 \otimes x_1 \bullet x_2 \ = \ - \   x_2x_1 \otimes x_1 \bullet x_2 \ = \ -\ (x_1 \succ x_2).$

For finish our proving, we have to mention only the following 

$\sum\limits_{\sigma \in \mathbb{S}_3} (-1)^\sigma \big(\alpha_\sigma^j (x_{\sigma(1)} \prec x_{\sigma(2)}) \prec x_{\sigma(3)}-
    \beta_\sigma^j x_{\sigma(1)}\prec(x_{\sigma(2)}\prec x_{\sigma(3)}) \big) \ = $

\quad \quad\quad $= \  \sum\limits_{\sigma \in \mathbb{S}_3} (-1)^\sigma 
\Big(\alpha_\sigma^j \big( (x_{\sigma(1)} x_{\sigma(2)})  x_{\sigma(3)} \big) \otimes \big(
 (x_{\sigma(1)} \bullet x_{\sigma(2)})  \bullet x_{\sigma(3)} \big)-$ 
 \begin{flushright}
$    \beta_\sigma^j  \big(  x_{\sigma(1)} (x_{\sigma(2)}  x_{\sigma(3)}) \big) \otimes 
 \big(  x_{\sigma(1)} \bullet (x_{\sigma(2)} \bullet x_{\sigma(3)}) \big) \Big)\  = $
 \end{flushright}

 \begin{flushright}
 $ = \ \big(x_1x_2\big)x_3 \otimes \Big(
 \sum\limits_{\sigma \in \mathbb{S}_3} \big(\alpha_\sigma^j (x_{\sigma(1)} \bullet  x_{\sigma(2)}) \bullet  x_{\sigma(3)} +
    \beta_\sigma^j x_{\sigma(1)}\bullet (x_{\sigma(2)}\bullet  x_{\sigma(3)}) \big) \ = \ 0.$
 \end{flushright}

\end{proof}

\subsection{Examples}

Below we give a list of $\mathcal M$anin anti-varieties for some known varieties of algebras.

\begin{longtable}{|l|rcl|rcl|}

\hline
 & \multicolumn{3}{|c|}{$\rm{Var}$} & 
 \multicolumn{3}{|c|}{$\mathcal{M}$a-$\rm{Var}$}\\

\hline\hline
Associative &$(xy)z $&$=$&$  x(yz)$ & 
$(xy)z $&$=$&$- \   x(yz)$ \\  
\hline
\hline
Shift associative &$(xy)z $&$=$&$  y(zx)$ & 
$(xy)z $&$=$&$- \   y(zx)$ \\  

\hline \hline
Leibniz &$(xy)z$&$=$&$ (xz)y +x(yz)  $  & 
$(xy)z$&$=$&$ -(xz)y  - x(yz) $ \\  

\hline\hline

Zinbiel & $(xy)z$&$=$&$ x (yz+zy)$ &
$(xy)z $&$=$&$ \ \   x (zy-yz)$ \\
\hline\hline

right symmetric  & 
$\big(x,y,z\big) $&$=$&$ \ \   \big(x,z,y\big) $ &
$\big\{x,y,z\big\} $&$=$&$ - \    \big\{x,z,y\big\}$ \\

\hline\hline

right alternative  & 
$\big(x,y,z\big) $&$=$&$ -\big(x,z,y\big) $ &
$\big\{x,y,z\big\} $&$=$&$ \big\{x,z,y\big\}$ \\

\hline
\hline

flexible   & 
$\big(x,y,z\big)  $&$=$&$    -  \big(z,y,x\big) $ &
$\big\{x,y,z \big\}  $&$=$&$      \big\{ z,y,x\big\} $ \\
\hline
\hline

Lie-admissible    & 
\multicolumn{3}{l|}{$\sum\limits_{\sigma \in \mathbb S_3}\big(-1\big)^{\sigma} \big(x_{\sigma(1)},x_{\sigma(2)},x_{\sigma(3)} \big) =0$} &
\multicolumn{3}{l|}{$\sum\limits_{\sigma \in \mathbb S_3}  \big\{x_{\sigma(1)},x_{\sigma(2)},x_{\sigma(3)} \big\} =0$} \\
\hline
\hline

assosymmetric  & 
$\big(x,y,z\big)  $&$=$&$   \big(x,z,y\big) $ &
$\big\{x,y,z\big\} $&$=$&$ - \ \big \{x,z,y\big\}$ \\

 & 
$\big(x,y,z\big)  $&$=$&$   \big(y,x,z\big) $ &
$\big\{x,y,z\big\}  $&$=$&$ - \  \big\{y,x,z\big\} $ \\

\hline\hline

Novikov  & 
$\big(x,y,z\big) $&$=$&$ \big(x,z,y\big)  $ &
$\big\{x,y,z\big\}  $&$=$&$ -  \ \big\{ x,z,y\big\}$ \\

 & 
$x(yz)  $&$=$&$ y(xz)  $ &
$x(yz)  $&$=$&$ - \ y(xz)  $ \\

\hline
\hline
bicommutative  & 
$x(yz)  $&$=$&$ y(xz)  $ &
$x(yz)  $&$=$&$ - \ y(xz)  $ \\

 & 
$(xy)z  $&$=$&$ (xz)y  $ &
$(xy)z  $&$=$&$ - \ (xz)y  $ \\

\hline
\hline
alternative  & 
$\big(x,y,z\big)  $&$=$&$   -  \big(x,z,y\big)$ &
$\big\{x,y,z\big\}  $&$=$&$  \big\{x,z,y\big\}$ \\

 & 
$\big(x,y,z\big)  $&$=$&$   -  \big(y,x,z\big)  $ &
$\big\{x,y,z\big\}  $&$=$&$   \big\{ y,x,z\big\}  $ \\

 \hline
\hline

$\big(-1,1\big)$- & 
$\big(x,y,z\big)  $&$=$&$  -  \big(x,z,y\big) $ &
$\big\{x,y,z\big\}  $&$=$&$     \big\{x,z,y\big\}$ \\
 & 
\multicolumn{3}{l|}{$\big(x,y,z\big)+\big(y,z,x\big)+\big(z,x,y\big)  \ = \ 0  $} &
\multicolumn{3}{l|}{$ \big\{x,y,z\big\}+\big\{y,z,x\big\}+\big\{z,x,y\big\} \ = \ 0 $} \\

 \hline
\hline

anti-pre-Lie  & 
$\big\{x,y,z\}  $&$=$&$   \big\{y,x,z\big\} $ &
$\big(x,y,z\big)  $&$=$&$   -  \big(y,x,z\big)$ \\
 & 
\multicolumn{3}{l|}{$[x,y]z +[y,z]x +[z,x]y  \ = \ 0  $} &
\multicolumn{3}{l|}{$ (x\circ y)z + (y\circ z)x +(z\circ x)y  \ = \ 0 $} \\

\hline

\end{longtable}

\noindent Remarks:
\begin{enumerate}
\item[$(\rm{A})$] The notion of   $\mathcal{M}$a-Leibniz algebras coincides with the notion of anti-Leibniz algebras, which were introduced in {\rm \cite{antileib}}.

\item[$(\rm{B})$] The notion of   $\mathcal{M}$a-Zinbiel algebras coincides with the notion of $\big(-1\big)$-Zinbiel algebras, which were introduced in {\rm \cite{dL}}.

 \item[$(\rm{C})$] The notion of  $\mathcal{M}$a-bicommutative  algebras coincides with the notion of anti-bicommutative   algebras, which were introduced in {\rm \cite{BBK}}.

 \item[$(\rm{D})$] The notion of  $\mathcal{M}$a-Novikov  algebras coincides with the notion of anti-Novikov  algebras, which were introduced in {\rm \cite{CMM}}.

\item[$(\rm{E})$] The notion of $\mathcal{M}$a-Lie admissible algebras coincides with the notion of Jacobi-Jordan-admissible  algebras introduced in {\rm \cite{BBMM}}.

\item[$(\rm{F})$] The notion of $\mathcal{M}$a-right alternative algebras coincides with the notion of pre-Jacobi-Jordan  algebras introduced in {\rm \cite{BBMM}}.

\item[$(\rm{G})$] The notion of $\mathcal{M}$a-anti-pre-Lie   coincides with the notion of 
anti-pre-Jacobi-Jordan   algebras introduced in {\rm \cite{AA26}}.

 \item[$(\rm{H})$] The notion of  $\mathcal{M}$a-Leibniz algebras 
and  $\mathcal{M}$a-Zinbiel algebras are special particular cases of generalized Leibniz and generalized Zinbiel algebras, which were introduced in {\rm \cite{MO14}}.
 
\end{enumerate}

\subsection{Symmetric affinization of 
$\mathcal{M}a$-Novikov algebras}

Let ${\rm A}$ be an algebra, then 
$\big(\mathfrak{aff}({\rm A}) = {\rm A} \otimes \mathbb{C}[t, t^{-1}], \circ \big)$ 
is called the affinization of ${\rm A}$ if 
\begin{center}
$at^n \circ bt^m = \big(m(ab)-n(ba)\big)t^{n+m-1}.$
\end{center}
It is known, 
that
$\mathfrak{aff}({\rm A})$ is Lie if and only if ${\rm A}$ is Novikov \cite{BN85}. 

Following this idea, we say that   
$\big(\mathfrak{aff}^+({\rm A}) = {\rm A} \otimes \mathbb{C}[t, t^{-1}], \circ \big)$ 
is called the symmetric affinization of ${\rm A}$ if 
\begin{center}
$at^n \circ bt^m = \big(m(ab)+n(ba)\big)t^{n+m-1}.$
\end{center}
The following proposition gives an analog of 
result from Balinskii and Novikov for Jacobi-Jordan algebras.

\begin{proposition}
Let ${\rm A}$ be an algebra, then 
$\mathfrak{aff}^+({\rm A})$ is Jacobi--Jordan if and only if
${\rm A}$ is $\mathcal{M}a$-Novikov.
\end{proposition}

\begin{proof}
Let $\mathfrak{aff}^+({\rm A})$ be Jacobi--Jordan and  $a_1, a_2, a_3 \in {\rm A},$ then 
    \begin{longtable}{lll}
$(a_1t^{n_1} \circ a_2 t^{n_2}) \circ a_3 t^{n_3}$&$=$&$
n_1(n_1+n_2-1)\ (a_1a_2)a_3+n_1n_3 \ a_3(a_1a_2)+$\\
\multicolumn{3}{r}{$n_2(n_1+n_2-1)\ (a_2a_1)a_3
+n_2n_3 \ a_3(a_2a_1);$}\\

$(a_2t^{n_2} \circ a_3 t^{n_3}) \circ a_1 t^{n_1}$&$=$&$
n_2(n_2+n_3-1)\ (a_2a_3)a_1+n_2n_1 \ a_1(a_2a_3)+$\\
\multicolumn{3}{r}{$n_3(n_2+n_3-1)\ (a_3a_2)a_1
+n_3n_1 \ a_1(a_3a_2);$}\\

$(a_3t^{n_3} \circ a_1 t^{n_1}) \circ a_2 t^{n_2}$&$=$&$
n_3(n_3+n_1-1)\ (a_3a_1)a_2+n_3n_2 \ a_2(a_3a_1)+$\\
\multicolumn{3}{r}{$n_1(n_3+n_1-1)\ (a_1a_3)a_2
+n_1n_2 \ a_2(a_1a_3).$}\\
    \end{longtable}
Summarizing the above relations and using the Jacobi--Jordan identity
\begin{enumerate}
    \item for $n_2=n_3=0,$ we have $(a_1a_2)a_3 = - (a_1a_2)a_3.$

    \item for $n_1=0,$ we have 
  
$n_2(n_2-1) \ (a_2a_1)a_3 + 
n_2n_3 \ a_3(a_2a_1) +
n_2(n_2+n_3-1) \ (a_2a_3)a_1 +
 $
\begin{flushright}
$n_3(n_2+n_3-1) \ (a_3a_2)a_1 + 
n_3(n_3-1) \ (a_3a_1)a_2 +
    n_3n_2 \ a_2(a_3a_1)=0.$
 \end{flushright}
   
The last is equivalent to 
$\big\{a_3,a_2,a_1 \big\} = - \big\{a_2,a_3,a_1\big\}.$\end{enumerate}
Hence, ${\rm A}$ is $\mathcal{M}a$-Novikov.
On the other side, it is easy to see that for each
$\mathcal{M}a$-Novikov, its symmetric affinization gives a Jacobi--Jordan algebra.
\end{proof}

\end{document}